\documentclass[11pt]{article}

\usepackage[letterpaper, margin=1in]{geometry}
\usepackage{cite}
\usepackage{amsmath, amssymb, amsthm}
\usepackage{graphicx}
\usepackage{float}
\usepackage{hyperref}
\usepackage{xcolor}

\theoremstyle{definition}
\newtheorem{definition}{Definition}[section]

\theoremstyle{plain}
\newtheorem{theorem}{Theorem}[section]
\newtheorem{proposition}[theorem]{Proposition}
\newtheorem{lemma}[theorem]{Lemma}
\newtheorem{corollary}[theorem]{Corollary}

\author{Miraj Samarakkody}
\title{Physics-Informed Neural Networks for Computing the Morse Index of the Critical Catenoid}
\date{\today}

\begin{document}
\maketitle

\begin{abstract}
	The Morse index of a free boundary minimal surface is encoded in its Jacobi--Steklov spectrum, and we test how faithfully a physics-informed neural network (PINN) reproduces that spectrum on a problem whose answer is already known in closed form. The benchmark is the critical catenoid in the unit ball $\mathbb{B}^3$, where it is well known that the Morse index equals $4$ and the nullity equals $2$. Separating the angular variable reduces the eigenvalue problem to a family of one-dimensional Robin problems on $[-T,T]$, one for each Fourier mode. A network that enforces the parity of each mode by construction, and carries the eigenvalue as a trainable parameter, returns the three eigenvalues below the stability threshold to within $10^{-6}$ to $10^{-4}$ of their exact values, with PDE residuals of order $10^{-4}$; assembling them recovers the index $4$ and the nullity $2$. We then track the spectrum along a one-parameter homotopy joining a flat reference operator to the catenoid Jacobi operator and identify the crossings at which the index changes. Since the critical catenoid is rigid, a fact we prove, this homotopy deforms operators rather than surfaces. We close by explaining how the same pipeline, with its one-dimensional solver replaced by a two-dimensional one, is poised to address genuinely geometric families in ellipsoidal balls, where the boundary curvature is no longer constant, and the Morse index is not yet known.
\end{abstract}



\section{Introduction and Motivation}

The Morse index of a minimal surface measures how unstable it is: it counts, with multiplicity, the independent directions in which the surface can be pushed to decrease area to second order. Classifying surfaces by this number has been a recurring theme. For closed minimal surfaces in the round sphere $\mathbb{S}^3$ the story is by now classical. Simons showed that the totally geodesic sphere has index $1$~\cite{simons01}, and Urbano proved that the Clifford torus is the only minimal surface of index $5$~\cite{urbano01}, a characterization that Marques and Neves would later put to use in their proof of the Willmore conjecture~\cite{marquesneves01}. A parallel theory exists for free boundary minimal surfaces (FBMS) in the unit ball $\mathbb{B}^3$, developed mainly by Fraser and Schoen~\cite{fraser01, fraser05}. There the equatorial disk has index $1$, and Tran~\cite[Theorem 1.3]{tran01} proved that the critical catenoid is the unique free boundary minimal annulus of index $4$, with nullity $2$. What matters for us is how Tran gets there: his method recasts the index as spectral data, pairing a fixed-boundary eigenvalue problem with a Dirichlet-to-Neumann (Jacobi--Steklov) map. The same index has since been recovered by other analytic routes as well~\cite{devyver01, smithzhou01}.

Our aim is not another proof of this number but a numerical route to the index that does not rely on a closed-form spectrum. The critical catenoid is the natural place to calibrate such a method: it is essentially the only free boundary minimal surface whose entire Jacobi--Steklov spectrum is known explicitly, so one can check, eigenvalue by eigenvalue, whether a solver returns the correct index. We use a physics-informed neural network (PINN), which represents the eigenfunction by a network and enforces the eigenvalue equation and its boundary conditions through the loss~\cite{raissi01}, with the eigenvalue itself carried as a trainable parameter. On this one-dimensional problem a PINN is not competitive with classical eigensolvers, and we claim no such thing; its value is that the same construction, unlike separation of variables, transfers without structural change to surfaces where no closed form is available. The target we have in mind is the free boundary problem in ellipsoidal balls, where the boundary curvature is non-constant, Tran's index formula no longer applies, and the Morse index is not known; the present paper builds and validates the pipeline on the catenoid as the step before that one.

This paper makes four main contributions. First, we develop a parity-constrained PINN that recovers the three Jacobi--Steklov eigenvalues of the critical catenoid below the stability threshold, and with them the Morse index $4$ and nullity $2$, to accuracy verified against the closed-form values (Sections~\ref{sec:pinn}--\ref{sec:index}). Second, we establish a spectral-gap statement isolating the eigenvalues responsible for the index (Corollary~\ref{cor:gap}). Third, we compute the spectral flow of a one-parameter operator homotopy from a flat reference operator to the catenoid Jacobi operator, identifying the eigenvalue crossings that change the index (Section~\ref{sec:flow}). Finally, we prove that the critical catenoid is rigid (Proposition~\ref{prop:rigid}); this shows the homotopy deforms operators rather than surfaces, and singles out ellipsoidal balls, where the index is open, as the setting for a genuinely geometric spectral flow (Section~\ref{sec:rigidity}).

To make this reduction precise we recall the Steklov framework of Fraser and Schoen~\cite{fraser01, fraser04, fraser05}, which underlies everything that follows.

The Steklov eigenvalues of $\Sigma$ arise from the harmonic extension of functions defined on the boundary $\partial \Sigma$. Given $h \in C^\infty(\partial \Sigma)$, let $\hat{h}$ denote its harmonic extension, that is, the unique solution of
\begin{equation}
	\begin{cases}
		\Delta \hat{h} = 0 & \text{ in } \Sigma          \\
		\hat{h} = h        & \text{ on } \partial \Sigma
	\end{cases}
\end{equation}
The Dirichlet-to-Neumann map associated with the Laplacian,
\begin{equation}
	\mathcal{L}_\Delta : C^\infty (\partial \Sigma) \to C^\infty (\partial \Sigma ),
\end{equation}
is defined by
\begin{equation}
	\mathcal{L}_\Delta h = \dfrac{\partial \hat{h}}{\partial \eta},
\end{equation}
where $\eta$ is the outward unit conormal along $\partial \Sigma$.

Since $\Delta$ is an elliptic self-adjoint operator, the harmonic extension exists and is unique, and $\mathcal{L}_\Delta$ is a non-negative self-adjoint operator with discrete spectrum
\begin{equation}
	0 = \xi_1 \leq \xi_2 \leq \dots \to \infty.
\end{equation}
The $\xi_i$ are called the \emph{Steklov eigenvalues} of $\Sigma$.

PINNs embed known mathematical or physical laws directly into the loss function of a neural network, enforcing them via automatic differentiation rather than classical numerical discretization~\cite{raissi01}. In the present context ``physics'' is the Jacobi operator eigenvalue problem: the network is trained to produce an eigenfunction while simultaneously satisfying the governing equation in the interior and the Robin-type boundary condition on the boundary circles of the catenoid. The eigenvalue itself is treated as a learnable parameter, updated by gradient descent alongside the network weights.

\section{Mathematical Setup}

Let $\Sigma^k \subset \Omega^{k+1}$ be a smooth, properly immersed, orientable free boundary minimal surface. Being two-sided, $\Sigma$ admits a globally defined smooth unit normal field $\nu$, so it suffices to consider normal variations $V = u\,\nu$ for a smooth function $u$. The second variation of volume of $\Sigma(t)$ at $\Sigma(0)$ defines the index bilinear form~\cite{maximo01}

\begin{equation}
	S(u,u) = \left.\dfrac{d^2}{dt^2}\,\mathrm{Vol}(\Sigma(t))\right|_{t=0} = \int_\Sigma \left( |\nabla^\Sigma u|^2 - \big(\mathrm{Ric}^\Omega(\nu, \nu) + |h|^2\big)\,u^2 \right) dA + \int_{\partial \Sigma} \langle \nabla_\nu^\Omega \nu, \eta\rangle\, u^2 \, ds,
\end{equation}

where superscripts indicate the ambient space in which each operation is taken. Here $|h|$ is the norm of the second fundamental form of $\Sigma \subset \Omega$, $\mathrm{Ric}^\Omega$ is the Ricci tensor of $\Omega$, and $\eta$ is the outward conormal of $\partial\Sigma$, which by the free boundary condition coincides with the outward unit normal of $\partial\Omega$. Consequently, the boundary integrand is the second fundamental form of $\partial\Omega$ evaluated on $\nu$:
\begin{equation}
	\langle \nabla_\nu^\Omega \nu, \eta \rangle = h^{\partial \Omega}(\nu, \nu),
\end{equation}
where $h^{\partial\Omega}(\cdot,\cdot)$ is taken with respect to the outward normal of $\partial\Omega \subset \Omega$.

\begin{definition}[Morse index and nullity]\label{def:index}
	The \emph{Morse index} of $\Sigma$ is the maximal dimension of a subspace $W \subseteq C^\infty(\Sigma)$ on which the index form $S$ is negative definite; that is,
	\begin{equation}
		\mathrm{Ind}(\Sigma) = \max\left\{ \dim W : W \subseteq C^\infty(\Sigma),\ S(u,u) < 0 \text{ for all } u \in W \setminus \{0\} \right\}.
	\end{equation}
	The \emph{nullity} of $\Sigma$ is the dimension of the null space of $S$, i.e.\ the space of $u \in C^\infty(\Sigma)$ for which $S(u,v) = 0$ for every $v \in C^\infty(\Sigma)$.
\end{definition}

Recall that the Jacobi operator is given by
\begin{equation}
	\mathcal{J} = \Delta_\Sigma + \mathrm{Ric}(\nu, \nu) + |h^\Sigma|^2.
\end{equation}
It is well known~\cite{maximo01, ambrozio01, sargent01} that the Morse index equals the number of negative eigenvalues, counted with multiplicity, of the following system:
\begin{equation}
	\begin{cases}
		\mathcal{J}u = -\lambda u                                             & \text{ in } \Sigma          \\
		\dfrac{\partial u}{\partial \eta} = - h^{\partial \Omega} (\nu, \nu)u & \text{ on } \partial \Sigma
	\end{cases}
\end{equation}

If we restrict to variations that fix the boundary, the boundary integral vanishes, leading to the following system:

\begin{equation}\label{eq:06}
	\begin{cases}
		\mathcal{J} u = - \lambda u & \text{ in } \Sigma          \\
		u \equiv 0                  & \text{ on } \partial \Sigma
	\end{cases}
\end{equation}

Because of the boundary condition, the number of negative eigenvalues of \eqref{eq:06} is generally smaller than the Morse index. The contribution of the boundary is captured by the following Dirichlet-to-Neumann map associated with the Jacobi operator.

Given a function $h \in C^\infty(\partial \Sigma)$, its \emph{Jacobi extension} $\hat{h}$ is a solution of the boundary value problem

\begin{equation}\label{eq:jacobi-ext}
	\begin{cases}
		\mathcal{J} \hat{h} = 0 & \text{ in } \Sigma          \\
		\hat{h} = h             & \text{ on } \partial \Sigma
	\end{cases}
\end{equation}

Unlike the harmonic extension, the Jacobi extension need not exist for every boundary datum, and when it does it is determined only up to a Jacobi field with vanishing boundary data; existence and the precise sense of uniqueness are governed by Lemma~\ref{lem:jacobi-ext} below. On the admissible subspace the construction is well posed, and the associated Dirichlet-to-Neumann map is defined by
\begin{equation}\label{eq:08}
	\mathcal{L}_\mathcal{J} h = \dfrac{\partial \hat{h}}{\partial \eta}.
\end{equation}
The operator $\mathcal{L}_\mathcal{J}$ has a discrete spectrum whose eigenvalues tend to infinity.

\begin{lemma}[{\cite[Lemma 2.5]{tran01}}]\label{lem:jacobi-ext}
	Let
	\begin{equation}
		\mathrm{J}_0^0 = \left\{ w \in C^\infty(\Sigma) : \mathcal{J} w = 0 \text{ in } \Sigma,\ w = 0 \text{ on } \partial \Sigma \right\}
	\end{equation}
	denote the space of Jacobi fields with vanishing boundary data, and let
	\begin{equation}
		D_\eta \mathrm{J}_0^0 = \left\{ \left.\dfrac{\partial w}{\partial \eta}\right|_{\partial \Sigma} : w \in \mathrm{J}_0^0 \right\}
	\end{equation}
	be the corresponding space of conormal derivatives along $\partial \Sigma$. Given $h \in C^\infty(\partial \Sigma)$, the Jacobi extension problem~\eqref{eq:jacobi-ext} admits a solution $\hat{h}$, unique up to the addition of an element of $\mathrm{J}_0^0$, if and only if
	\begin{equation}
		\int_{\partial \Sigma} h\, b \, ds = 0 \qquad \text{for all } b \in D_\eta \mathrm{J}_0^0.
	\end{equation}
	Equivalently, $h \in (D_\eta \mathrm{J}_0^0)^\perp \subset C^\infty(\partial \Sigma)$ with respect to the $L^2$ inner product on $\partial \Sigma$.
\end{lemma}

\begin{theorem}[{\cite[Theorem 1.2]{tran01}}]\label{thm:tran}
	Let $\Sigma^k \subset \Omega^{k+1}$ be a smooth, properly immersed, orientable FBMS for which $h^{\partial \Omega}(\nu, \nu) = -c$ is constant. Then, counting multiplicities,
	\begin{enumerate}
		\item the Morse index of $\Sigma$ equals the number of negative eigenvalues of the fixed-boundary problem~\eqref{eq:06} plus the number of eigenvalues of the Dirichlet-to-Neumann map~\eqref{eq:08} that are strictly less than $c$; and
		\item the nullity of $\Sigma$ equals the dimension of the eigenspace of~\eqref{eq:08} associated with the eigenvalue $c$.
	\end{enumerate}
\end{theorem}

The critical catenoid is conjectured to be the unique (up to congruence) free boundary minimal annulus in $\mathbb{B}^3$ of genus zero with two boundary components, a conjecture due to Fraser~\cite{fraser03} with partial progress by McGrath~\cite{mcgarth01}. It is parametrized by
\begin{equation}
	X(t, \theta) = c (\cosh t \cos \theta, \cosh t \sin \theta, t),
\end{equation}
with $t \in [-T, T]$ and $\theta \in [0, 2\pi]$, where the constants $T$ and $c$ are determined by the free-boundary condition. (Figure \ref{fig:catenoid})

\begin{equation}
	\cosh T = T \sinh T \text{ and } c = \dfrac{1}{T \cosh T}
\end{equation}
Numerically, $T \approx 1.1997$, $\cosh T \approx 1.81$, $\sinh T \approx 1.51$, and $\tanh T \approx 0.83$. The unit outward normal is \begin{equation}
	\nu = \dfrac{1}{\cosh t}(\cos \theta, \sin \theta, -\sinh t),
\end{equation} and the squared norm of the second fundamental form is $|h|^2 = \dfrac{2}{c^2 \cosh^4 t}$.

\begin{figure}[H]
	\centering
	\includegraphics[width=0.6\textwidth]{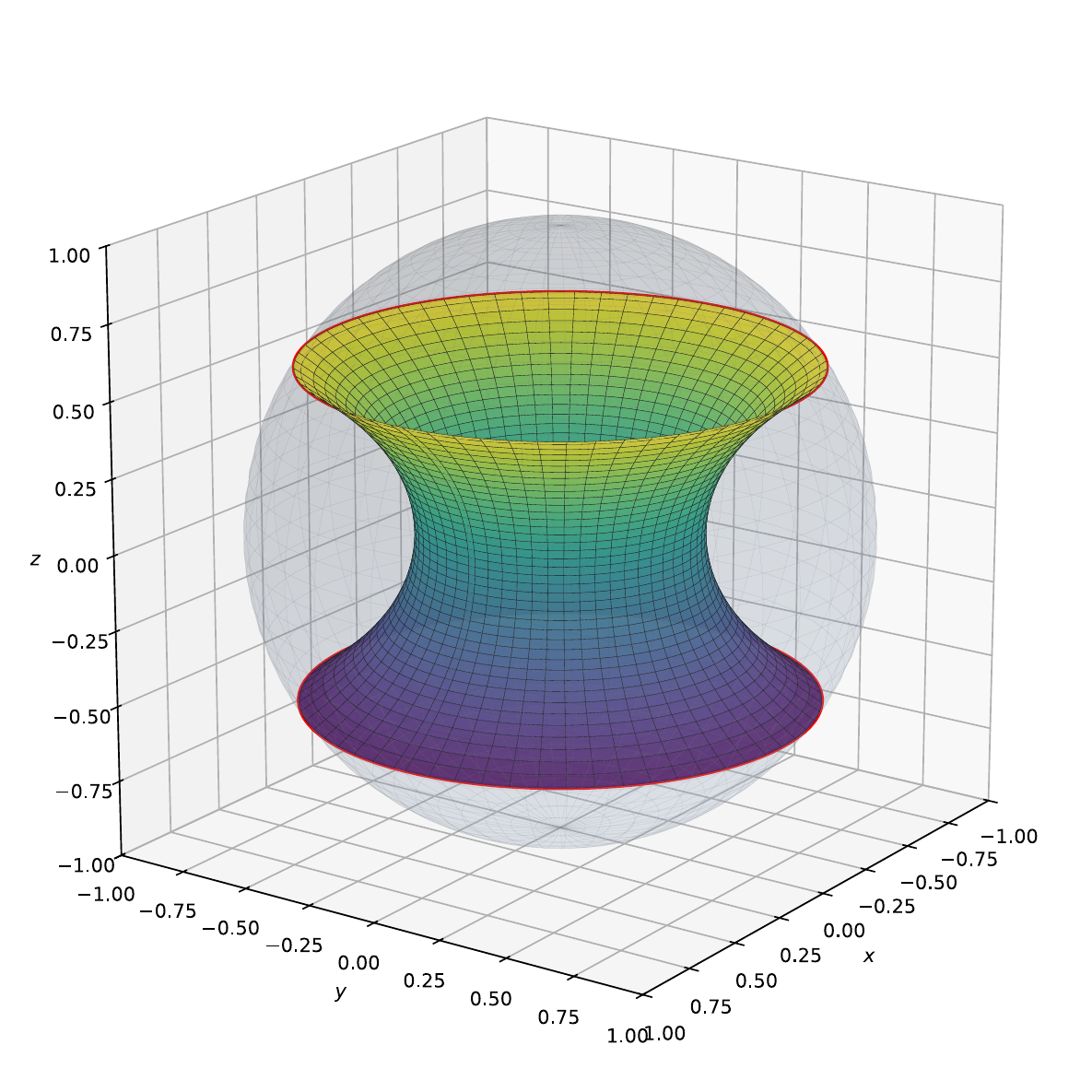}
	\caption{The critical catenoid in $\mathbb{B}^3$.}
	\label{fig:catenoid}
\end{figure}

\subsection{The Jacobi-Steklov Eigenvalue Problem}

For the critical catenoid the ambient space is flat, so $\mathrm{Ric}\equiv 0$ and the Jacobi operator is $\mathcal{J} = \Delta_\Sigma + |h|^2$. The induced metric is conformally flat,
\begin{equation}
	ds^2 = c^2 \cosh^2 t\,(dt^2 + d\theta^2),
\end{equation}
with conformal factor $\rho = c\cosh t$, so that $\Delta_\Sigma = \rho^{-2}(\partial_{tt} + \partial_{\theta\theta})$. Substituting $|h|^2 = 2/(c^2\cosh^4 t)$ gives
\begin{equation}\label{eq:cat-jacobi}
	\mathcal{J} = \frac{1}{c^2 \cosh^2 t}\left[ \partial_{tt} + \partial_{\theta\theta} + \frac{2}{\cosh^2 t} \right].
\end{equation}

Because the catenoid is rotationally symmetric about the $x_3$-axis, $\mathcal{J}$ commutes with $\partial_\theta^2$, and separation of variables $u(t,\theta) = f(t)g(\theta)$ with $g(\theta) \in \{\cos n\theta, \sin n\theta\}$ is valid. The conformal factor $\rho^{-2}$ does not affect the kernel of $\mathcal{J}$, so the Jacobi extension equation $\mathcal{J}\hat{h} = 0$ in \eqref{eq:jacobi-ext} reduces, for the $n$-th Fourier mode, to the ordinary differential equation on $[-T,T]$,
\begin{equation}\label{eq:cat-ode}
	f''(t) + \left( \frac{2}{\cosh^2 t} - n^2 \right) f(t) = 0.
\end{equation}

It remains to translate the Dirichlet-to-Neumann eigenvalue equation $\mathcal{L}_\mathcal{J}h = \delta h$ into a boundary condition for $f$, where $\delta$ denotes a Jacobi-Steklov eigenvalue. Along the boundary circle $t = +T$ the outward conormal is the unit vector $\eta = \rho^{-1}\partial_t$, so for a separated function the conormal derivative is
\begin{equation}
	\frac{\partial u}{\partial \eta} = \frac{1}{c\cosh T}\, f'(+T)\, g(\theta).
\end{equation}
The free-boundary relation $\cosh T = T\sinh T$ together with $c = 1/(T\cosh T)$ yields the identity $(c\cosh T)^{-1} = T$. Hence on the circle $t = +T$ the equation $\mathcal{L}_\mathcal{J}h = \delta h$ becomes $T f'(+T) = \delta f(+T)$; on the circle $t = -T$ the outward conormal is $-\rho^{-1}\partial_t$, and the same computation gives $-T f'(-T) = \delta f(-T)$. The Jacobi-Steklov problem is therefore the family of Robin eigenvalue problems
\begin{equation}\label{eq:cat-robin}
	\begin{cases}
		f''(t) + \left( \dfrac{2}{\cosh^2 t} - n^2 \right) f(t) = 0, & t \in (-T, T), \\[6pt]
		T f'(+T) = \delta f(+T), \quad -T f'(-T) = \delta f(-T).     &
	\end{cases}
\end{equation}

For the unit ball the boundary $\partial\mathbb{B}^3$ is the unit sphere, whose second fundamental form with respect to the outward normal is $h^{\partial\mathbb{B}^3}(\nu,\nu) = -1$. The threshold constant of Theorem~\ref{thm:tran} is therefore $1$ (not to be confused with the catenoid scale $c = 1/(T\cosh T)$), and the stability threshold for the Jacobi-Steklov eigenvalues is $\delta = 1$. By that theorem the Morse index equals the number of non-positive eigenvalues of the fixed-boundary problem~\eqref{eq:06} plus the number of eigenvalues $\delta < 1$ of~\eqref{eq:cat-robin}, counted with multiplicity, while the nullity is the dimension of the eigenspace at $\delta = 1$. The fixed-boundary contribution can be read off directly: a short computation gives
\begin{equation}\label{eq:zeta}
	\zeta := \langle X, \nu \rangle = \frac{c}{\cosh t}\bigl(\cosh t - t\sinh t\bigr),
\end{equation}
which is strictly positive on $(-T,T)$ and vanishes at $t = \pm T$ by the free-boundary condition. Thus $\zeta$ is a Jacobi field with zero boundary data, hence the first eigenfunction of~\eqref{eq:06} with eigenvalue $0$, and there are no negative eigenvalues; the fixed-boundary count is exactly $1$. The remaining task, counting the Jacobi-Steklov eigenvalues of~\eqref{eq:cat-robin} below $1$, is what we carry out numerically. Tran computed the spectrum of~\eqref{eq:cat-robin} in closed form~\cite[Theorem 4.2]{tran01}; the eigenvalues relevant to the index are collected in Table~\ref{tab:eig}.

\section{Physics-Informed Neural Network Method}\label{sec:pinn}

We solve the Robin eigenvalue problem~\eqref{eq:cat-robin} with a physics-informed neural network. The idea is straightforward: a network stands in for the eigenfunction $f$, and the eigenvalue $\delta$ rides along as a single trainable scalar, optimized together with the network weights. Training then pushes down three competing terms at once; the interior ODE residual, the Robin boundary residual, and a normalization that keeps the network away from the trivial solution $f \equiv 0$.

\subsection{Architecture and parity constraint}

The core network $N \colon \mathbb{R} \to \mathbb{R}$ is an ordinary fully connected net: four hidden layers of width $64$, with $\tanh$ activations. The one twist is parity. Each Fourier mode of~\eqref{eq:cat-robin} is either even or odd in $t$, and rather than leave the network to discover this on its own, we build it in. Writing $N$ for the raw network, we represent the eigenfunction as
\begin{equation}\label{eq:parity}
	f(t) = N(t) + N(-t) \quad \text{(even mode)}, \qquad f(t) = N(t) - N(-t) \quad \text{(odd mode)}.
\end{equation}
Without this, an unconstrained network drifts into mixed-parity configurations, spurious minima that drive down the residual while corrupting the eigenvalue, and the optimization is badly conditioned; with it, those minima simply vanish. This is the single change that makes everything else work. The eigenvalue $\delta$ starts near a target value and is then learned by gradient descent through the boundary loss: the optimizer finds it, we do not put it in by hand. Table~\ref{tab:arch} collects the settings.

\begin{table}[ht]
	\centering
	\begin{tabular}{lll}
		\hline
		\textbf{Component} & \textbf{Specification}      & \textbf{Notes}                 \\
		\hline
		Input              & $t \in [-T, T]$             & one collocation coordinate     \\
		Hidden layers      & $4 \times 64$, $\tanh$      & fully connected                \\
		Parity             & $f = N(t) \pm N(-t)$        & exact even/odd by construction \\
		Eigenvalue         & $\delta$ (trainable scalar) & learned by gradient descent    \\
		Optimizer          & Adam, lr $= 2\times10^{-3}$ & cosine annealing to $10^{-5}$  \\
		Collocation        & $600$ uniform points        & interior of $[-T, T]$          \\
		Epochs             & $9000$--$15000$             & per mode                       \\
		\hline
	\end{tabular}
	\caption{PINN architecture and training configuration for the one-dimensional Robin eigenvalue problem~\eqref{eq:cat-robin}.}
	\label{tab:arch}
\end{table}

\subsection{Loss function}

The total loss is $\mathcal{L} = \mathcal{L}_{\mathrm{pde}} + 5\,\mathcal{L}_{\mathrm{bc}} + 2\,\mathcal{L}_{\mathrm{norm}}$, with
\begin{align}
	\mathcal{L}_{\mathrm{pde}}  & = \frac{1}{M}\sum_{i=1}^{M} \left[ f''(t_i) + \left( \frac{2}{\cosh^2 t_i} - n^2 \right) f(t_i) \right]^2, \\
	\mathcal{L}_{\mathrm{bc}}   & = \bigl[ T f'(+T) - \delta f(+T) \bigr]^2 + \bigl[ -T f'(-T) - \delta f(-T) \bigr]^2,                      \\
	\mathcal{L}_{\mathrm{norm}} & = \left( \int_{-T}^{T} f^2 \, dt - 1 \right)^2,
\end{align}
where the $\{t_i\}_{i=1}^{M}$ are the $M = 600$ collocation points. The derivatives $f'$ and $f''$ come for free from automatic differentiation, and the normalization integral is evaluated by the trapezoidal rule on the same grid. We fixed the relative weights $5$ and $2$ once, at the outset, and left them untouched across every mode, and there was no per-mode tuning.

The solver is implemented in Python~3.12. The network, automatic differentiation, and the Adam optimizer are provided by PyTorch~\cite{pytorch01}; the spectral-flow ground truth of Section~\ref{sec:flow} uses the ODE-integration and root-finding routines of SciPy~\cite{scipy01} together with NumPy~\cite{numpy01}, and the figures are produced with Matplotlib~\cite{matplotlib01}.

\section{Morse Index of the Critical Catenoid}\label{sec:index}

\subsection{Recovered eigenvalues}

We trained one network per mode for the three Jacobi-Steklov eigenvalues at or below the threshold. With the parity constraint~\eqref{eq:parity} in place and $\delta$ learned by gradient descent, the eigenvalues are recovered to between $10^{-6}$ and $10^{-4}$ of their analytic values; the error is mode-dependent, the $n=1$ even mode converging to $\delta = -0.99990$ against the exact value $-1$. The PDE residuals at convergence are of order $10^{-4}$. Nothing about $\delta$ is supplied in advance; it is discovered through the loss, and the small errors, which vary from mode to mode, are simply what a finite-width network trained for finitely many epochs achieves on these smooth, low-frequency eigenfunctions. Table~\ref{tab:eig} collects the comparison, and Figure~\ref{fig:index} shows that the PINN eigenfunctions coincide with the analytic ones to plotting accuracy.

\begin{table}[ht]
	\centering
	\begin{tabular}{lccccc}
		\hline
		\textbf{Mode} & \textbf{Exact } $\delta$ & \textbf{PINN } $\delta$ & $|\textbf{error}|$ & \textbf{PDE resid} & $\delta < 1$?                          \\
		\hline
		$n=0$ odd     & $+0.43923$               & $+0.43923$              & $2.6\times10^{-6}$ & $6.9\times10^{-5}$ & \textcolor{red}{Yes ($\times 1$)}      \\
		$n=1$ even    & $-1.00000$               & $-0.99990$              & $9.7\times10^{-5}$ & $5.2\times10^{-4}$ & \textcolor{red}{Yes ($\times 2$)}      \\
		$n=1$ odd     & $+1.00000$               & $+1.00000$              & $1.2\times10^{-6}$ & $2.8\times10^{-4}$ & \textcolor{green!50!black}{No (null.)} \\
		$n=2$ even    & $+2.00000$               & ---                     & ---                & ---                & \textcolor{green!50!black}{No}         \\
		$n=3$ even    & $+3.42087$               & ---                     & ---                & ---                & \textcolor{green!50!black}{No}         \\
		\hline
	\end{tabular}
	\caption{Jacobi-Steklov eigenvalues of the critical catenoid. Exact values are from Tran~\cite[Theorem 4.2]{tran01}; PINN values are learned by gradient descent. Modes with $\delta < 1$ (red) contribute to the Morse index; $\times 2$ indicates that $\cos\theta$ and $\sin\theta$ give independent eigenfunctions. Networks were trained for $n=0$ and $n=1$; the $n\geq 2$ modes lie above threshold by the analytic formula.}
	\label{tab:eig}
\end{table}

\begin{figure}[H]
	\centering
	\includegraphics[width=0.8\textwidth]{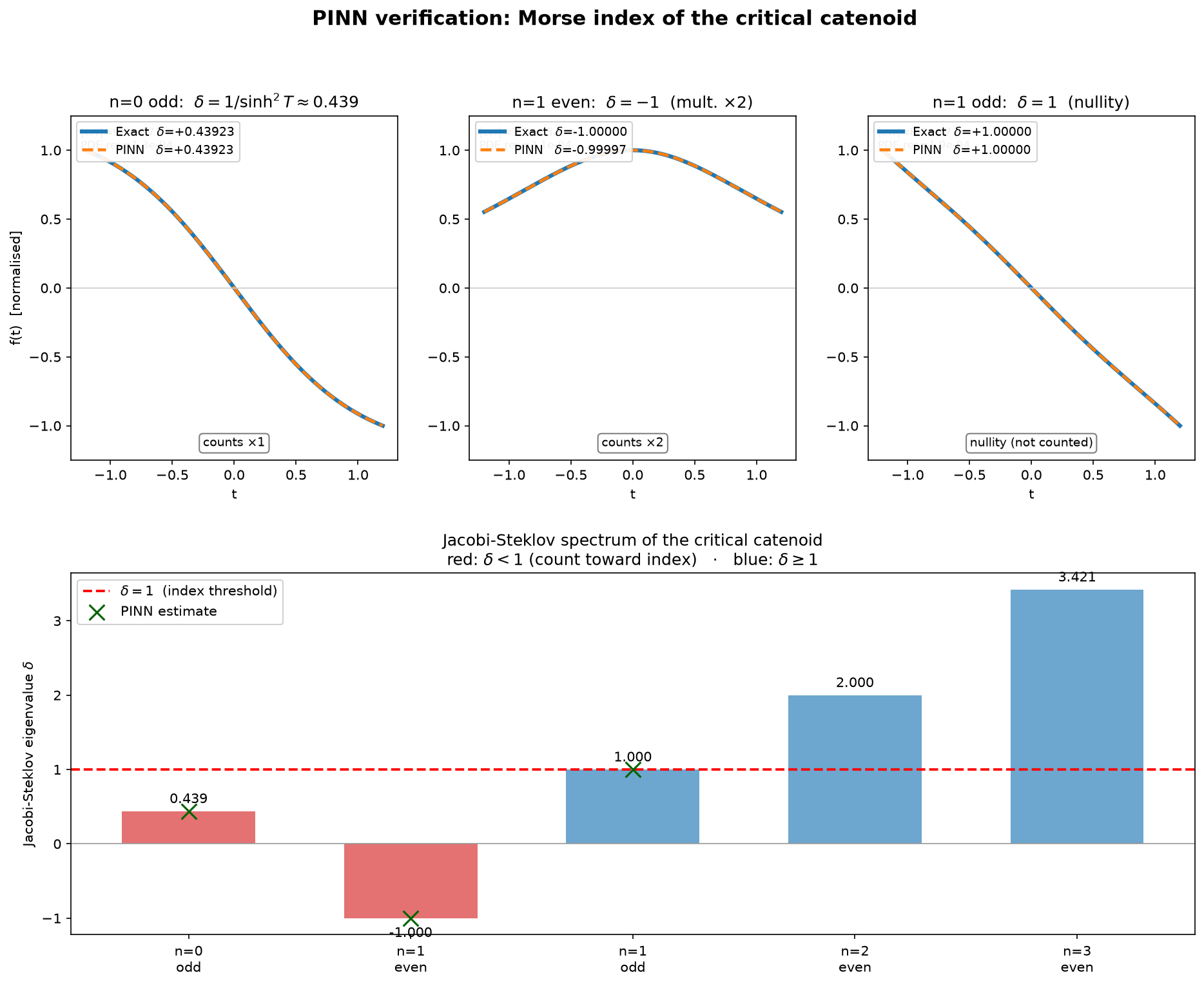}
	\caption{Top: analytic (solid) versus PINN (dashed) eigenfunctions for the three lowest Jacobi-Steklov modes, with per-mode eigenvalue error and PDE residual shown. The $n=0$ odd mode ($\delta\approx 0.439$) and the $n=1$ even mode ($\delta = -1$, multiplicity $2$) lie below the threshold $\delta = 1$ and count toward the index; the $n=1$ odd mode ($\delta = 1$) is the nullity mode. Bottom: the Jacobi-Steklov spectrum for $n=0$ to $3$; red bars mark $\delta < 1$, blue bars $\delta \geq 1$, and crosses the PINN estimates. The annotation displays the index computation in the form of Theorem~\ref{thm:tran}.}
	\label{fig:index}
\end{figure}

\subsection{Index, nullity, and a spectral gap}

Collecting the two contributions, the fixed-boundary count is $1$ (the Jacobi field $\zeta$ of~\eqref{eq:zeta}), and the Jacobi-Steklov eigenvalues below $1$ are $\delta = 1/\sinh^2 T \approx 0.439$ (multiplicity $1$) and $\delta = -1$ (multiplicity $2$), three in total. Therefore
\begin{equation}
	\mathrm{Ind}(\Sigma) = \#\{\lambda \leq 0 \text{ of }\eqref{eq:06}\} + \#\{\delta < 1 \text{ of } \mathcal{L}_\mathcal{J}\} = 1 + 3 = 4,
\end{equation}
and the nullity is the dimension of the $\delta=1$ eigenspace of $\mathcal{L}_\mathcal{J}$, namely $2$, spanned by the two rotational Jacobi fields. Both agree with Tran~\cite[Theorem 1.3]{tran01} and with the independent analytic computations of Smith--Zhou~\cite{smithzhou01} and Devyver~\cite{devyver01}. The computation also exhibits a clean spectral gap, which we record.

\begin{corollary}[spectral gap]\label{cor:gap}
	The Jacobi-Steklov operator $\mathcal{L}_\mathcal{J}$ of the critical catenoid has exactly three eigenvalues below $1$, counted with multiplicity: $\delta = 1/\sinh^2 T \approx 0.4392$ (multiplicity $1$) and $\delta = -1$ (multiplicity $2$). No eigenvalue lies in the open interval $(0.4392, 1)$; the next eigenvalue is $\delta = 1$ (the nullity, multiplicity $2$), followed by $\delta \approx 2$. In particular $\sigma(\mathcal{L}_\mathcal{J}) \cap (0.4392, 1) = \varnothing$.
\end{corollary}

\begin{proof}
	This is immediate from Tran's closed-form spectrum~\cite[Theorem 4.2]{tran01}: the $n=0$ odd eigenvalue $1/\sinh^2 T$ is the largest eigenvalue below $1$, and the $n\geq 2$ eigenvalues exceed $1$, the smallest being $\delta \approx 2$ at $n=2$. The gap is confirmed numerically by the ODE-shooting solver of Section~\ref{sec:flow} to tolerance $10^{-8}$.
\end{proof}

The gap is the free-boundary analogue of the spectral separation underlying the index characterizations of minimal surfaces in $\mathbb{S}^3$~\cite{urbano01, marquesneves01}: a non-equatorial free boundary minimal annulus in $\mathbb{B}^3$ has index at least $4$, and the critical catenoid is the unique annulus realizing equality.

\section{Spectral Flow of an Operator Homotopy}\label{sec:flow}

The computation of Section~\ref{sec:index} fixes a single operator. We now ask how the spectrum moves as the operator is deformed. The Morse index changes as Jacobi-Steklov eigenvalues cross the threshold $\delta = 1$ or leave the Dirichlet-to-Neumann spectrum; tracking these events reveals the spectral flow of the operator.

\subsection{The homotopy}

We deform the coefficients of the Robin problem~\eqref{eq:cat-robin} by a parameter $s \in [0,1]$,
\begin{equation}\label{eq:homotopy}
	T(s) = 0.3 + s\,(T - 0.3), \qquad \mu(s) = s,
\end{equation}
where $T \approx 1.1997$ and $\mu$ controls the strength of the potential $2\mu/\cosh^2 t$. At $s = 1$ the operator is the catenoid Jacobi-Steklov operator of Section~\ref{sec:index}; at $s = 0$ the potential vanishes and the operator reduces to the flat Laplacian on a short strip. We emphasize that the intermediate values $0 < s < 1$ are operator-theoretic interpolations, not free boundary minimal surfaces: by the rigidity proved in Section~\ref{sec:rigidity}, no FBMS corresponds to the parameters $(T(s), \mu(s))$ for $0 < s < 1$. What the computation reveals is how many eigenvalue crossings, and of which Fourier modes, separate the flat reference spectrum from the catenoid spectrum.

\subsection{Numerical results}

At each $s$ we compute the Jacobi-Steklov eigenvalues by high-accuracy ODE shooting (relative tolerance $10^{-8}$) for the modes $n = 0,1,2$, and we validate with PINN spot-checks at $s = 0.2, 0.5, 0.8, 1.0$. Figure~\ref{fig:flow} shows the spectral flow. The eigenvalues move continuously with $s$, and several cross the threshold $\delta = 1$. The resulting Morse index is piecewise constant, with the plateaus of Table~\ref{tab:plateau}.

\begin{table}[ht]
	\centering
	\begin{tabular}{p{2.4cm}cp{4.7cm}p{3.9cm}}
		\hline
		\textbf{Parameter interval} & \textbf{Index} & \textbf{Index contributions ($\delta < 1$, with mult.)}            & \textbf{Triggering event}                                \\
		\hline
		$s = 0$                     & $5$            & $n{=}0$ even, $n{=}1$ even ($\times 2$), $n{=}2$ even ($\times 2$) & flat reference state                                     \\
		$(0, 0.4)$                  & $7$            & the above, plus $n{=}0$ odd and the fixed-boundary field $\zeta$   & $n{=}0$ odd drops below $\delta{=}1$; $\zeta$ enters     \\
		$[0.4, 1)$                  & $5$            & $n{=}0$ even, $n{=}1$ even ($\times 2$), $n{=}0$ odd, $\zeta$      & $n{=}2$ even pair exits above $\delta{=}1$               \\
		$s = 1$                     & $4$            & $n{=}1$ even ($\times 2$), $n{=}0$ odd, $\zeta$                    & $n{=}0$ even eigenvalue leaves the spectrum ($u(T){=}0$) \\
		\hline
	\end{tabular}
	\caption{Morse index plateaus along the operator homotopy~\eqref{eq:homotopy}. The contributions list the Dirichlet-to-Neumann modes with $\delta < 1$ (with multiplicity; $n = 0$ modes are simple, $n \geq 1$ modes doubled by $\cos n\theta, \sin n\theta$) together with the fixed-boundary field $\zeta$; their total is the index. The index takes the values $5$ at $s = 0$, $7$ on $(0, 0.4)$, and $5$ on $[0.4, 1)$, dropping to $4$ precisely at the catenoid endpoint $s = 1$, where the $n{=}0$ even eigenfunction satisfies the free-boundary condition $u(T) = 0$ and the mode leaves the Dirichlet-to-Neumann spectrum. The value at $s = 1$ reproduces the catenoid index $4$ of Section~\ref{sec:index}, with contributions $n{=}1$ even ($\times 2$), $n{=}0$ odd, and $\zeta$.}
	\label{tab:plateau}
\end{table}

The index is non-monotone along the homotopy, rising to $7$ before settling at the catenoid value $4$. The transient rise comes from modes that sit below threshold while the strip is short and the potential weak but leave the count as the catenoid geometry develops: the $n{=}2$ even pair crosses back above $\delta = 1$ near $s = 0.4$, while the $n{=}0$ even eigenvalue decreases monotonically and diverges to $-\infty$ as $s \to 1$. The divergence is exact and meaningful: the $n{=}0$ even eigenfunction is the $s$-deformation of the Jacobi field $\zeta$ of~\eqref{eq:zeta}, whose boundary value $u(T)$ vanishes \emph{only} when $\cosh T = T\sinh T$, i.e.\ only at the catenoid $s = 1$; there the mode leaves the Dirichlet-to-Neumann spectrum and is counted instead by the fixed-boundary term. The index therefore drops to its catenoid value $4$ exactly when the free-boundary condition is met. Only the endpoints, the flat strip at $s = 0$ and the critical catenoid at $s = 1$, are spectrally meaningful in the geometric sense; what the computation establishes unambiguously is the count and Fourier type of the crossings, which an analytic argument would reproduce only through detailed spectral estimates.

\begin{figure}[H]
	\centering
	\includegraphics[width=0.95\textwidth]{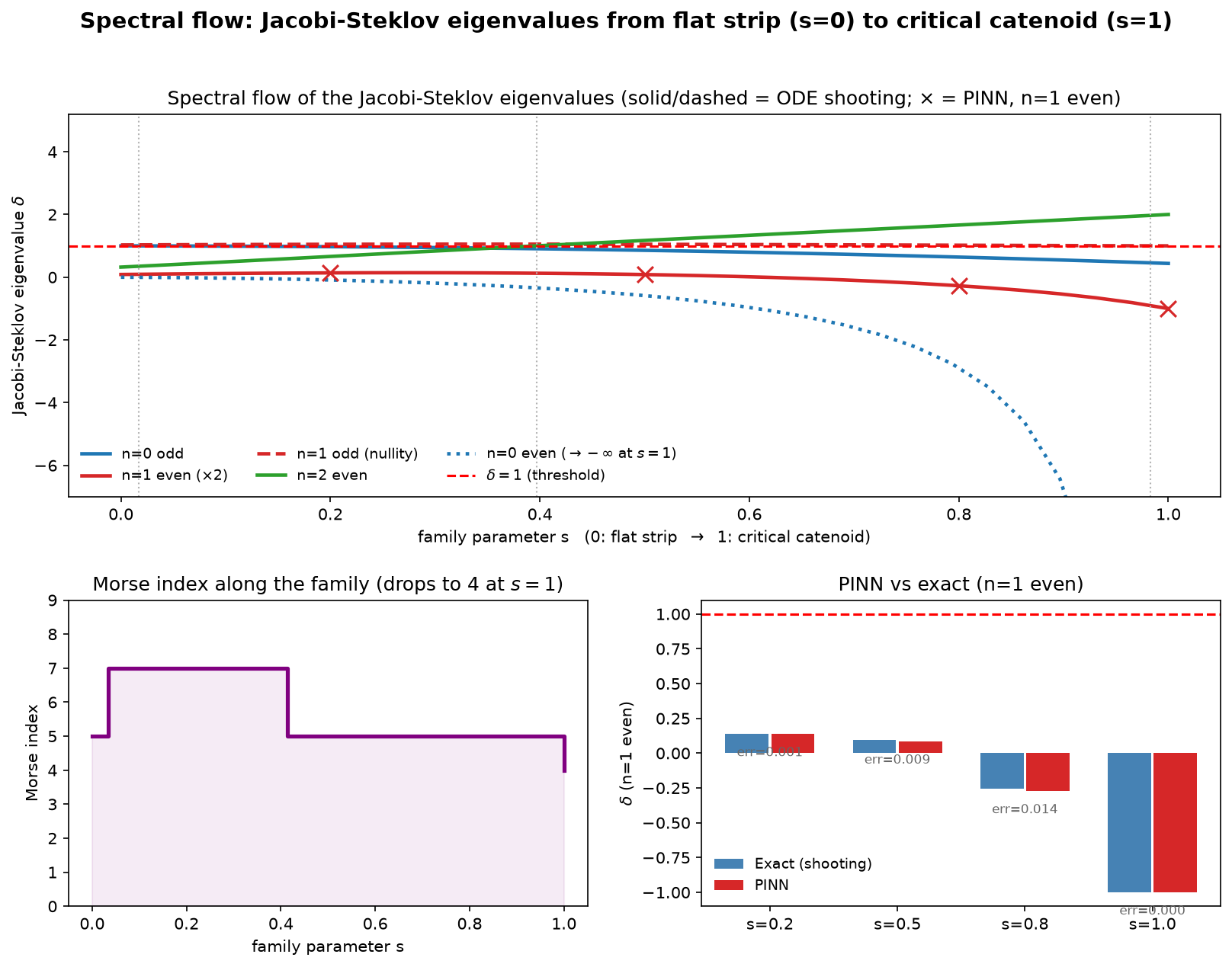}
	\caption{Spectral flow along the operator homotopy~\eqref{eq:homotopy}. Top: Jacobi-Steklov eigenvalues $\delta(s)$ for the Fourier modes $n = 0,1,2$ (ODE shooting, solid/dashed) with PINN spot-checks (crosses) on the $n=1$ even mode; the red dashed line is the threshold $\delta = 1$ and the dotted verticals mark index jumps. Lower left: the Morse index staircase. Lower right: PINN versus exact eigenvalue at four parameter values, with the error growing as the $n=1$ even eigenvalue moves away from $0$.}
	\label{fig:flow}
\end{figure}

\section{Rigidity of the Catenoid and Genuinely Geometric Families}\label{sec:rigidity}

\subsection{The critical catenoid is rigid}

We justify the claim, used in Section~\ref{sec:flow}, that no free boundary minimal surface corresponds to the intermediate parameters of the homotopy. The argument is a short rigidity computation.

\begin{proposition}\label{prop:rigid}
	The free-boundary condition determines $T$ uniquely as the root of $\cosh T = T\sinh T$; the radius condition then fixes $c = 1/(T\cosh T)$; and minimality fixes the Jacobi potential to $2/\cosh^2 t$, that is, $\mu = 1$. The triple $(T, c, \mu)$ admits no free parameter.
\end{proposition}

\begin{proof}
	\emph{Free-boundary condition.} The surface meets $\partial\mathbb{B}^3$ orthogonally if and only if the radial vector $X$ is tangent to $\Sigma$ along $\partial\Sigma$, equivalently $\langle X, \nu\rangle = 0$ there. From the parametrization a direct computation gives $\langle X, \nu\rangle = (c/\cosh t)(\cosh t - t\sinh t)$ as in~\eqref{eq:zeta}, so at $t = T$ the condition is $\cosh T = T\sinh T$.

	\emph{Uniqueness of $T$.} Let $g(T) = \cosh T - T\sinh T$. Then $g(0) = 1 > 0$ and $g'(T) = -T\cosh T < 0$ for $T > 0$, so $g$ is strictly decreasing and has the single root $T \approx 1.19968$. Thus $T$ is determined uniquely.

	\emph{Determination of $c$ and $\mu$.} The condition $|X| = 1$ at $t = T$ reads $c^2(\cosh^2 T + T^2) = 1$; using $\cosh T = T\sinh T$ this simplifies to $c = 1/(T\cosh T)$. Finally, the Jacobi potential is $2/\cosh^2 t$, since
	\begin{equation}
		|h|^2 \cdot (c^2\cosh^2 t) = \frac{2}{c^2\cosh^4 t}\cdot(c^2\cosh^2 t) = \frac{2}{\cosh^2 t},
	\end{equation}
	that is, $\mu = 1$ with no free parameter. Hence a family with $\mu(s) = s$ and $T(s)\neq T$ satisfies none of these conditions for $0 < s < 1$.
\end{proof}

\subsection{Why varying the radius is trivial}

A natural attempt at a geometric family deforms the ambient ball radius. This fails. If $\Sigma \subset \mathbb{B}^3_R$ is free boundary minimal, then $\Sigma/R \subset \mathbb{B}^3_1$ is too, because both minimality and orthogonal intersection are scale-invariant; the free-boundary condition $\cosh T = T\sinh T$ is unchanged (the radius cancels), so $T$ does not move. The entire Jacobi-Steklov problem rescales covariantly, the boundary mean curvature becomes $1/R$ and the threshold becomes $\delta < 1/R$, and the count rebalances exactly. Hence the index is $4$ for every $R$, and the radius family carries no spectral flow.

\subsection{Ellipsoidal balls}

A genuinely nontrivial family must deform the ambient domain non-conformally. The natural choice is the ellipsoidal ball with semi-axes $(1, 1, 1+\varepsilon)$. For $\varepsilon \neq 0$ the boundary second fundamental form $h^{\partial\Omega}(\nu, \nu)$ is no longer constant along $\partial\Sigma$. Theorem~\ref{thm:tran} assumes precisely that this curvature is constant, so it no longer applies, and we are not aware of an index characterization valid in the non-constant case. Supplying one, for instance an eigenvalue-comparison argument that brackets the index between auxiliary operators built from the infimum and supremum of $-h^{\partial\Omega}(\nu, \nu)$, is an open problem; in particular, the Morse index in such ellipsoidal balls is not determined by existing theory. The Jacobi-Steklov problem on such a surface is also genuinely two-dimensional, with no separation of variables, so the one-dimensional Robin solver of Section~\ref{sec:pinn} would be replaced by a two-dimensional PINN over the surface domain, while the loss structure, the eigenvalue-as-parameter device, and the index-counting procedure carry over unchanged. As $\varepsilon$ increases from $0$, the index may jump as eigenvalues cross the now $\varepsilon$-dependent threshold. Tracking these crossings would produce the first genuinely geometric spectral flow for free boundary minimal surfaces; establishing the comparison estimate that would make such a count rigorous, and carrying out the computation, is the direction we intend to pursue.

\section{Discussion}

Two points deserve emphasis. The first is the parity constraint~\eqref{eq:parity}, without which the problem is essentially intractable for a network: an unconstrained net settles into mixed-parity configurations that drive down the interior residual while badly mis-estimating the eigenvalue. Once parity is built in, gradient descent alone recovers the eigenvalue to within $10^{-4}$, since the eigenfunctions are smooth and low-frequency and well within the reach of a modest network. We make no claim that this beats a classical eigensolver on a one-dimensional problem; the point is that the same construction carries over, unchanged in structure, to the two-dimensional setting where separation of variables is no longer available.

The second is the spectral flow of Section~\ref{sec:flow}. It is not geometric, but it is still informative: it counts and classifies the crossings separating a trivial reference operator from the catenoid operator, information an analytic argument would extract only with careful estimates. We are deliberate in calling it an operator homotopy rather than a deformation through surfaces, because Proposition~\ref{prop:rigid} shows there are no intermediate surfaces to deform through. The same observation is what tells us that a genuinely geometric spectral flow has to be sought elsewhere, in the ellipsoidal setting.

What Sections~\ref{sec:pinn}--\ref{sec:rigidity} provide, then, is a single pipeline: a parity-constrained eigenvalue solver calibrated against the one case with a closed-form answer, a spectral-flow procedure, and the index-counting rule of Theorem~\ref{thm:tran}. None of it is special to one dimension, so it should apply to genuinely geometric families as soon as the surfaces are available as parametrized domains. Section~\ref{sec:rigidity} points to where to look first: the ellipsoidal family, where the boundary curvature varies, Tran's formula no longer applies directly, and the index is not yet known.

\newpage

\newpage

\bibliographystyle{plain}
\bibliography{references}

\end{document}